\documentclass{amsart}
\usepackage{amssymb}
\usepackage{mathrsfs}
\usepackage{amsmath}
\usepackage{graphicx}
\usepackage{color}
\usepackage{tikz}
\usetikzlibrary{cd}
\usepackage{hyperref}

\newcommand{\Z}{\mathbb{Z}}
\newcommand{\C}{\mathbb{C}}
\newcommand{\bT}{\mathbb{T}}

\newcommand{\bc}{\mathbf{c}}
\newcommand{\bg}{\mathbf{g}}
\newcommand{\bw}{\mathbf{w}}
\newcommand{\bv}{\mathbf{v}}
\newcommand{\bP}{\mathbf{P}}
\newcommand{\be}{\mathbf{e}}

\newcommand{\sT}{\mathsf{T}}
\newcommand{\sICFT}{\mathsf{T}'}

\newcommand{\GL}{\mathrm{GL}}

\newcommand{\Db}{D^{\mathrm{b}}}
\newcommand{\cF}{\mathcal{F}}
\newcommand{\cO}{\mathcal{O}}
\newcommand{\cI}{\mathcal{I}}
\newcommand{\cJ}{\mathcal{J}}
\newcommand{\cK}{\mathcal{K}}

\newcommand{\id}{\mathrm{id}}
\newcommand{\Rep}{\mathrm{Rep}}
\newcommand{\Hom}{\mathrm{Hom}}
\newcommand{\Fl}{\mathrm{Fl}}
\newcommand{\IC}{\mathrm{IC}}

\newcommand{\tE}{\widetilde{E}}

\newcommand{\Delc}{\mathsf{Del}^{\scriptscriptstyle\searrow}}
\newcommand{\Dell}{\mathsf{Del}_{\scriptscriptstyle\nearrow}}
\newcommand{\Top}{\mathsf{Top}}
\newcommand{\Raise}{\mathsf{Raise}}
\newcommand{\Lower}{\mathsf{Lower}}
\newcommand{\cupc}{\cup^{\scriptscriptstyle\searrow}}
\newcommand{\cupl}{\cup_{\scriptscriptstyle\nearrow}}
\newcommand{\sa}{\mathsf{a}}
\newcommand{\sA}{\mathsf{A}}
\newcommand{\sB}{\mathsf{B}}

\DeclareMathOperator{\udim}{\underline{\dim}}
\DeclareMathOperator{\tr}{tr}

\DeclareMathOperator{\spn}{span}

\newcommand{\stri}[6]{\hbox{$\vcenter{\hbox{\tiny\begin{tikzpicture}[scale=0.16]
\draw (0,3.3) -- (5,0.3) -- (0,-2.7) -- cycle;
\node (11) at (.70,1.5) {#1};
\node (21) at (.70,0) {#4};
\node (31) at (.70,-1.5) {#6};
\node (12) at (2.20,.75) {#2};
\node (22) at (2.20,-.75) {#5};
\node (13) at (3.70,0) {#3};
\end{tikzpicture}}}$}}

\tikzstyle{mathy}=[nodes={%
   execute at begin node=$,%
   execute at end node=$%
  }]

\newcommand{\triddots}{\rotatebox[origin=c]{-27}{$\cdots$}}
\newcommand{\triudots}{\rotatebox[origin=c]{27}{$\cdots$}}

\theoremstyle{plain}
\newtheorem{theorem}{Theorem}[section]
\newtheorem{lemma}[theorem]{Lemma}
\newtheorem{proposition}[theorem]{Proposition}
\newtheorem{corollary}[theorem]{Corollary}

\theoremstyle{definition}
\newtheorem{definition}[theorem]{Definition}

\theoremstyle{remark}
\newtheorem{remark}[theorem]{Remark}

\numberwithin{equation}{section}

\title[Combinatorial Fourier transform]{Combinatorics of Fourier transforms for\\
type A quiver representations}

\author{Pramod N. Achar}
\address{Department of Mathematics, Louisiana State University, Baton Rouge, LA 70803.}
\email{pramod@math.lsu.edu}
\author{Maitreyee C. Kulkarni}
\address{Department of Mathematics, Louisiana State University, Baton Rouge, LA 70803.}
\email{mkulka2@lsu.edu}
\author{Jacob P. Matherne}
\address{Department of Mathematics and Statistics, University of Massachusetts, Amherst, MA 01003.}
\email{matherne@math.umass.edu}

\thanks{P.A. received support from NSF Grant No.~DMS-1500890. M.K. received support from NSF Grant No.~DMS-1601862.  J.M. received support from a Department of Education GAANN fellowship (Grant No.~P200A120001). }


\begin{document}

\begin{abstract}
We describe two new combinatorial algorithms (using the language of ``triangular arrays'') for computing the Fourier transforms of simple perverse sheaves on the moduli space of representations of an equioriented quiver of type $A$. (A rather different solution to this problem was previously obtained by Knight--Zelevinsky.) Along the way, we also show that the closure partial order and the dimensions of orbits have especially concise descriptions in the language of triangular arrays.
\end{abstract}

\maketitle

\section{Introduction}
\label{sec:intro}

Let $Q_n$ be the following quiver, with $n$ vertices and $n-1$ arrows:
\begin{equation}\label{eqn:quiver}
\bullet\longrightarrow\bullet\longrightarrow \cdots \longrightarrow \bullet
\end{equation}
Given a dimension vector $\bw \in \Z^n_{\ge 0}$, 
let $E(\bw)$ be the moduli space of representations of $Q_n$ of dimension vector $\bw$.  (See Section~\ref{sec:prelim} for additonal background, definitions, and notation.)

This paper is the result of the authors' attempts to do exercises with perverse sheaves on $E(\bw)$, and specifically to compute Fourier--Sato transforms by hand. These exercises led to combinatorial objects  called \emph{triangular arrays}.  
Using the language of triangular arrays, we describe:
\begin{enumerate}
\item the closure partial order on orbits in $E(\bw)$ (Theorem~\ref{thm:po})
\item a dimension formula for orbits in $E(\bw)$ (Theorem~\ref{thm:dim})
\item two new combinatorial algorithms for computing Fourier--Sato transforms of simple perverse sheaves (Theorem~\ref{thm:main} and Corollary~\ref{cor:main})
\end{enumerate}
All of these problems have been previously solved in the language of \emph{multisegments}, also called \emph{Kostant partitions}~\cite{adf, lusztig, kz:rqa, b:cbqfm} (see also~\cite{bg:psrs}).  Nevertheless, we hope to convince the reader that the language of triangular arrays is worth studying:
\begin{itemize}
\item The closure partial order is especially easy in this language (it is the ``chutewise dominance order''), and the dimension formula is also very concise.
\item The combinatorics of the Fourier--Sato transform in this paper looks very different from the ``multisegment duality'' of~\cite{kz:rqa}.  (Indeed, we were unable to find an elementary relationship between the two.)  Perhaps our algorithms will be useful in situations where~\cite{kz:rqa} is difficult to apply.
\end{itemize}
For examples of triangular arrays, see Figures~\ref{fig:partialorder} and~\ref{fig:fourier}. Figure~\ref{fig:partialorder} shows the partial order and dimensions of orbits for the dimension vector $\bw = (3,3,3)$, and Figure~\ref{fig:fourier} shows the involution on this set of orbits induced by the Fourier--Sato transform.

\begin{figure}
\[
\begin{tikzcd}[row sep=small]
18 &&&\stri003033 \\
17 && \stri012123 \ar[ur, dash] && \stri102033 \ar[ul,dash] \\
16 &&& \stri102123 \ar[ul, dash] \ar[ur, dash] \\
15 &&& \stri111123 \ar[u, dash] \\
14 & \stri021213 \ar[uuur, dash] &&&& \stri201033 \ar[uuul, dash] \\
13 && \stri111213 \ar[ul, dash] \ar[uur, dash] && \stri201123 \ar[ur, dash] \ar[uul, dash] \\
11 && \stri120213 \ar[u, dash] && \stri210123 \ar[u,dash] \\
10 &&& \stri201213 \ar[uul, dash] \ar[uur, dash] \\
 9 & \stri030303 \ar[uuuu, dash] && \stri210213 \ar[uul, dash] \ar[u, dash] \ar[uur, dash] && \stri300033 \ar[uuuu, dash] \\
 8 & \stri120303 \ar[u, dash] \ar[uuur, dash] &&&& \stri300123 \ar[u, dash] \ar[uuul, dash] \\
 5 && \stri210303 \ar[ul, dash] \ar[uur, dash] && \stri300213 \ar[ur, dash] \ar[uul, dash] \\
 0 &&& \stri300303 \ar[ul, dash] \ar[ur, dash]
\end{tikzcd}
\]
\caption{Partial order and dimensions for $\bw = (3,3,3)$}\label{fig:partialorder}
\end{figure}

\begin{figure}
\[
\begin{tikzcd}[row sep=small]
&&\stri003033
\ar[ddddddddddd, leftrightarrow, in=30, out=-30, distance=35] \\
& \stri012123 \ar[ddddddddd, leftrightarrow, in=150, out=-150, distance=30] && \stri102033 \ar[ddddddddd, leftrightarrow, in=30, out=-30, distance=30]  \\
&& \stri102123 \ar[ddddd, leftrightarrow, in=150, out=-150, distance=25] \\
&& \stri111123 \ar[ddddd, leftrightarrow, in=30, out=-30, distance=25]  \\
\stri021213  \ar[ddddd, leftrightarrow, in=165, out=-165, distance=25] &&&& \stri201033 \ar[ddddd, leftrightarrow, in=15, out=-15, distance=25]  \\
& \stri111213 \ar[loop, leftrightarrow, in=-165, out=165, distance=15] && \stri201123 \ar[loop, leftrightarrow, in=-15, out=15, distance=15] \\
& \stri120213 \ar[loop, leftrightarrow, in=-165, out=165, distance=15] && \stri210123 \ar[loop, leftrightarrow, in=-15, out=15, distance=15] \\
&& \stri201213 \\
\stri030303 \ar[loop, leftrightarrow, in=-165, out=165, distance=15] && \stri210213 && \stri300033 \ar[loop, leftrightarrow, in=-15, out=15, distance=15] \\
\stri120303 &&&& \stri300123 \\
& \stri210303 && \stri300213 \\
&& \stri300303
\end{tikzcd}
\]
\caption{Fourier--Sato transforms for $\bw = (3,3,3)$}\label{fig:fourier}
\end{figure}
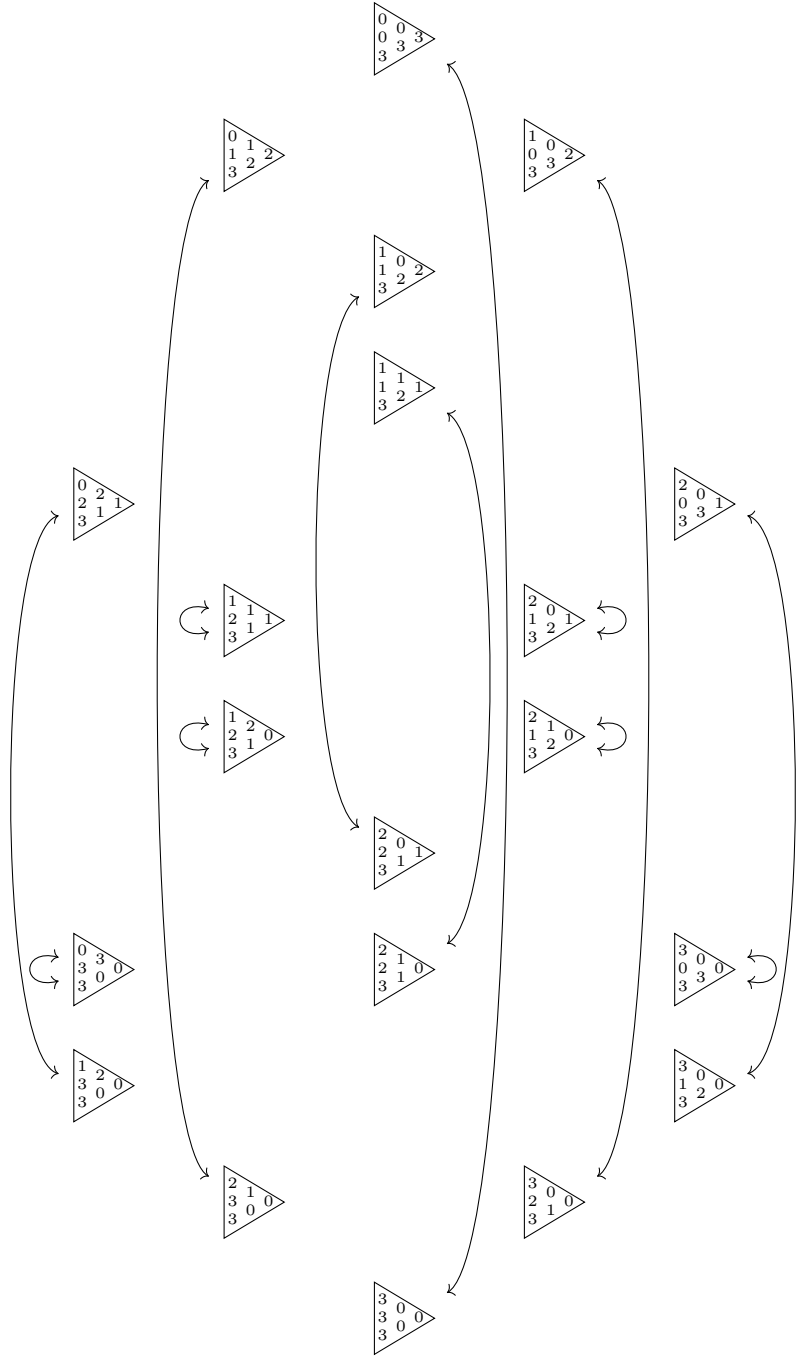

The paper is organized as follows: Section~\ref{sec:prelim} defines triangular arrays, and fixes notation related to quiver representations.  In Sections~\ref{sec:partialorder} and~\ref{sec:dimension}, we determine the closure partial order and the dimensions of orbits in terms of triangular arrays.

The main new content of the paper is in Sections~\ref{sec:operations} and~\ref{sec:geometry}.  Section~\ref{sec:operations} contains the definitions of two combinatorial operations on triangular arrays, denoted by $\sT$ and $\sICFT$.  That section also contains the proof that $\sT$ and $\sICFT$ are both bijections.  In Section~\ref{sec:geometry}, we prove that both $\sT$ and $\sICFT$ compute the Fourier--Sato transform for simple perverse sheaves on $E(\bw)$.  (In particular, the geometry shows that the maps $\sT$ and $\sICFT$ coincide. We do not know a combinatorial proof of this fact.)

\subsection*{Acknowledgments}

We are grateful to Pierre Baumann, Tom Braden, Thomas Br\"ustle, Lutz Hille, Ivan Mirkovi\'{c}, Laura Rider, Ralf Schiffler, and Catharina Stroppel for helpful conversations while this work was in progress.

\section{Notation and preliminaries}
\label{sec:prelim}

\subsection{Triangular arrays}\label{sec:triarr}

Let $\bw = (w_1, \ldots, w_n)$ be an $n$-tuple of nonnegative integers.  Given such a $\bw$, we define $\bP(\bw)$ to be the following set of collections of nonnegative integers:
\[
\bP(\bw) = \left\{ (y_{ij})_{1 \le i \le n, 1 \le j \le n-i+1} \,\Big|\,
\begin{array}{c}
\text{$\sum_{j=1}^{n-i+1} y_{ij} = w_i$ for all $i$, and} \\
\text{$y_{ij} \ge y_{i-1,j+1}$ for all $i$ and $j$}
\end{array} \right\}.
\]
An element $Y \in \bP(\bw)$ is called a \emph{triangular array of size $n$}. It can be drawn as follows:
\begin{equation}\label{eqn:tri-defn}
Y = \vcenter{\hbox{\scalebox{0.75}{\begin{tikzpicture}[scale=0.5,mathy]
\draw (0,3.67) -- (6.67,0) -- (0,-3.67) -- cycle;
\node (11) at (.70,2.25) {y_{11}};
\node at (.70,0) {\vdots};
\node (41) at (.70,-2.25) {y_{n1}};
\node (12) at (2.20,1.5) {y_{12}};
\node (22) at (2.20,0) {\vdots};
\node (32) at (2.20,-1.5) {y_{n-1,2}};
\node (13) at (3.70,.75) {\triddots};
\node (23) at (3.70,-.75) {\triudots};
\node (14) at (5.20,0) {y_{1n}};
\end{tikzpicture}}}}
\end{equation}
We will refer to portions of this diagram as \emph{columns}, \emph{chutes}, and \emph{ladders}:
\[
\begin{array}{c}
\text{$j$th} \\
\text{column:}
\end{array}
\vcenter{\hbox{\scalebox{0.6}{\begin{tikzpicture}[scale=0.5,mathy]
\draw (0,4.37) -- (8.16,0) -- (0,-4.37) -- cycle;
\node (12) at (2.20,2.25) {y_{1j}};
\node (22) at (2.20,0.75) {y_{2j}};
\node (32) at (2.20,-0.75) {\vdots};
\node (42) at (2.20,-2.25) {y_{n-j+1,j}};
\end{tikzpicture}}}}
\quad
\begin{array}{c}
\text{$i$th} \\
\text{chute:}
\end{array}
\vcenter{\hbox{\scalebox{0.6}{\begin{tikzpicture}[scale=0.5,mathy]
\draw (0,4.37) -- (8.16,0) -- (0,-4.37) -- cycle;
\node (21) at (.70,1.5) {y_{i1}};
\node (22) at (2.20,0.75) {y_{i2}};
\node (23) at (3.70,0) {\triddots};
\node (24) at (5.20,-.75) {y_{i,n+1-i}};
\end{tikzpicture}}}}
\quad
\begin{array}{c}
\text{$k$th} \\
\text{ladder:}
\end{array}
\vcenter{\hbox{\scalebox{0.6}{\begin{tikzpicture}[scale=0.5,mathy]
\draw (0,4.37) -- (8.16,0) -- (0,-4.37) -- cycle;
\node (41) at (.70,-1.5) {y_{k1}};
\node (32) at (2.20,-0.75) {\triudots};
\node (23) at (3.70,0) {y_{2,k-1}};
\node (14) at (5.20,.75) {y_{1k}};
\end{tikzpicture}}}}
\]
With these notions, we can rephrase the definition of $\bP(\bw)$ as follows: it is the set of diagrams of nonnegative integers as in~\eqref{eqn:tri-defn} that:
\begin{itemize}
\item have chute-sums given by $\bw$, and
\item are weakly decreasing (from left to right) along ladders.
\end{itemize}
For $Y \in \bP(\bw)$, we call $\bw$ the \emph{dimension vector} of $Y$, and we write
\[
\udim(Y) = \bw.
\]
Now let $Y = (y_{ij})$ and $Y' = (y'_{ij})$ be two elements of $\bP(\bw)$.  We equip $\bP(\bw)$ with a partial order $\le_\bc$ by declaring that
\begin{equation}\label{eqn:comb-order-defn}
Y \le_\bc Y' 
\qquad
\text{if for all $i$ and $j$,}
\qquad
\sum_{k=1}^j y_{ik} \ge \sum_{k=1}^j y'_{ik}.
\end{equation}
The condition~\eqref{eqn:comb-order-defn} resembles the usual dominance order on partitions, but each inequality involves only entries from a single chute.  For this reason, we call $\le_\bc$ the ``chutewise dominance order.''

\subsection{Moduli spaces of quiver representations}
\label{ss:quiver-rep}

Recall from Section~\ref{sec:intro} that $Q_n$ denotes the quiver~\eqref{eqn:quiver} with $n$ vertices and $n-1$ arrows.  Let $\Rep(Q_n)$ denote the category of finite-dimensional complex representations of $Q_n$.  Given an object 
\[
M = (M_1 \xrightarrow{x_1} M_2 \xrightarrow{x_2} \cdots \xrightarrow{x_{n-1}} M_n)
\]
in $\Rep(Q_n)$, we denote by $\udim M$ its dimension vector:
\[
\udim M = (\dim M_1, \dim M_2, \ldots, \dim M_n) \in \Z_{\ge 0}^n.
\]

Given $\bw = (w_1, \ldots, w_n) \in \Z_{\ge 0}^n$, let $E(\bw)$ be the moduli space of representations of $Q_n$ with dimension vector $\bw$.  Explicitly, we put
\[
E(\bw) = \Hom(\C^{w_1}, \C^{w_2}) \times \Hom(\C^{w_2},\C^{w_3}) \times \cdots \times \Hom(\C^{w_{n-1}},\C^{w_n}).
\]
Given $x = (x_1, \ldots, x_{n-1}) \in E(\bw)$, let $M(x)$ denote the quiver representation
\[
M(x) = (\C^{w_1} \xrightarrow{x_1} \C^{w_2} \xrightarrow{x_2} \cdots \xrightarrow{x_{n-1}} \C^{w_n}).
\]
(The point $x$ and the object $M(x)$ consist of the same data, but we think of $x$ as a point in an algebraic variety, and $M(x)$ as an object of an abelian category.)

The variety $E(\bw)$ is just an affine space of dimension $w_1w_2 + w_2w_3 + \cdots + w_{n-1}w_n$.  It is equipped with an action of the group
\[
G(\bw) = \GL(w_1) \times \GL(w_2) \times \cdots \times \GL(w_n)
\]
given by the formula $(g_1,\ldots,g_n) \cdot (x_1,\ldots,x_{n-1}) = (g_2x_1g_1^{-1}, \ldots, g_nx_{n-1}g_{n-1}^{-1})$.  Two points $x, y \in E(\bw)$ lie in the same $G(\bw)$-orbit if and only if $M(x)$ and $M(y)$ are isomorphic objects of $\Rep(Q_n)$.

Let us recall the classification of indecomposable objects in $\Rep(Q_n)$.  For $k = 1, \ldots, n$, let $\be_k$ be the dimension vector
\[
\be_k = (\underbrace{0,\ldots,0}_{\text{$k-1$ entries}},1,\underbrace{0,\ldots,0}_{\text{$n-k$ entries}}).
\]
Then, for $1 \le i \le j \le n$, let
\[
\gamma_{ij} = \be_i + \be_{i+1} + \cdots + \be_j.
\]
The $\gamma_{ij}$ can be identified with the positive roots in a root system of type $A_n$.  (The $\be_k$ are then identified with the simple roots.)

Gabriel's theorem~\cite{gab:ud1} says that the indecomposable objects in $\Rep(Q_n)$ are classified by their dimension vectors, and the vectors that occur as dimension vectors of indecomposable objects are precisely the positive roots.  Given integers $1 \le i \le j \le n$, let $R_{ij}$ be the quiver representation given by
\[
R_{ij} = 0 \to \cdots \to 0 \to \underset{\text{vertex $i$}}{\C} \xrightarrow{\id} \cdots \xrightarrow{\id} \underset{\text{vertex $j$}}{\C} \to 0 \to \cdots \to 0.
\]
Its dimension vector is $\gamma_{ij}$.  The $R_{ij}$ exhaust the isomorphism classes of indecomposables.

Consider the set
\[
\textstyle
B(\bw) = \left\{ (b_{ij})_{1 \le i \le j \le n} \mid \sum b_{ij} \gamma_{ij} = \bw \right\}.
\]
Gabriel's theorem implies that there is a canonical bijection
\begin{equation}\label{eqn:gabriel}
\{G(\bw)\text{-orbits on } E(\bw)\} \stackrel{1-1}{\longleftrightarrow} B(\bw).
\end{equation}

\begin{lemma}\label{lem:pb-bij}
There is a bijection
$\nu: \bP(\bw) \stackrel{1-1}{\longrightarrow} B(\bw)$.
\end{lemma}
\begin{proof}
Given $Y \in \bP(\bw)$, let $\nu(Y)$ be the element of $B(\bw)$ given by
\[
\nu(Y)_{ij} = y_{i,j-i+1} - y_{i-1,j-i+2},
\]
where the second term is understood to be $0$ if $i = 1$.  
Conversely, given $b = (b_{ij}) \in B(\bw)$, let $\bar\nu(b)$ be the triangular array in $\bP(\bw)$ given by
\[
\bar\nu(b)_{ij} = \sum_{1 \le h \le i} b_{h,i+j-1}.
\]
Straightforward computations show that $\nu$ and $\bar\nu$ are both well-defined, and that they are inverse to each other.
\end{proof}

For $\bw = (w_1,\ldots,w_n)$, let $\bw^* = (w_n,\ldots,w_1)$ be the reverse of $\bw$.

\begin{corollary}\label{cor:samecard}
The sets $\bP(\bw)$ and $\bP(\bw^*)$ have the same cardinality.
\end{corollary}
\begin{proof}
This follows from the fact that there is a bijection $B(\bw) \to B(\bw^*)$ given by $(b_{ij}) \mapsto (b_{n-j+1,n-i+1})$. 
\end{proof}

\subsection{Orbits}
\label{ss:orbits}

Combining~\eqref{eqn:gabriel} and Lemma~\ref{lem:pb-bij}, we obtain a bijection between $\bP(\bw)$ and the set of $G(\bw)$-orbits in $E(\bw)$.  For $Y \in \bP(\bw)$, let
\[
\cO_Y \subset E(\bw)
\]
be the corresponding $G(\bw)$-orbit.  Let us write down a concrete representative of this orbit.  

\begin{lemma}
Let $x \in E(\bw)$, and let $Y = (y_{ij}) \in \bP(\bw)$.  The following are equivalent:
\begin{enumerate}
\item $x \in \cO_Y$.\label{it:orbit}
\item Each $\C^{w_i}$ admits a basis\label{it:jordan}
\begin{equation}\label{eqn:model-basis}
\{ u_{ij}^{(k)} \mid \text{$1 \le j \le n-i+1$, $1 \le k \le y_{ij}$} \}
\end{equation}
such that $x_i: \C^{w_i} \to \C^{w_{i+1}}$ is given by
\begin{equation}\label{eqn:model-repn}
x_i(u_{ij}^{(k)}) = 
\begin{cases}
u_{i+1,j-1}^{(k)} & \text{if $j > 1$,} \\
0 & \text{if $j = 1$.}
\end{cases}
\end{equation}
\end{enumerate}
\end{lemma}

Note that the set in~\eqref{eqn:model-basis} does indeed consist of exactly $w_i$ elements.  We will call a basis in which~\eqref{eqn:model-repn} holds a \emph{Jordan basis of type $Y$}, by analogy with Jordan normal form for matrices.

\begin{proof}
We will first show that part~\eqref{it:jordan} implies part~\eqref{it:orbit}.  Assume that~\eqref{eqn:model-repn} holds.  Let $(b_{ij}) = \nu(Y) \in B(\bw)$.  To show that $x \in \cO_Y$, we must show that the representation $M(x)$ contains exactly $b_{ij}$ copies of $R_{ij}$ as direct summands, for all $i$ and $j$ such that $1 \le i \le j \le n$.  Fix such an $i$ and $j$. Also fix an integer $k$ such that
\[
y_{i-1,j-i+2}+1 \le k \le y_{i,j-i+1}.
\]
(If $i = 1$, then $y_{i-1,j-i+2}$ should be understood to be $0$.) Let $N^h_k \subset \C^{w_h}$ be the subspace given by
\[
N^h_k = 
\begin{cases}
0 & \text{if $h < i$ or $h > j$,} \\
\spn\{u_{h,j-h+1}^{(k)}\} & \text{if $i \le h \le j$.}
\end{cases}
\]
It can be checked using~\eqref{eqn:model-repn} that $N_k = \bigoplus_h N^h_k$ is a subrepresentation of $x$, and that it is isomorphic to $R_{ij}$.  On the other hand, the span of the basis elements from~\eqref{eqn:model-basis} that are not included in $N_k$ is also a subrepresentation, so $N_k$ is a direct summand.  The number of choices for $k$ is $y_{i,j-i+1} - y_{i-1,j-i+2} = b_{ij}$, so we have shown that $x$ contains at least $b_{ij}$ copies of $R_{ij}$ as direct summands.  The total dimension vector of the summands we have produced is already equal to $\bw$, so in fact $x$ contains exactly $b_{ij}$ copies of $R_{ij}$.

Suppose now that $x \in \cO_Y$. Define a new representation $z \in E(\bw)$ by choosing some basis as in~\eqref{eqn:model-basis}, and then defining the linear maps $z_i: \C^{w_i} \to \C^{w_{i+1}}$ using the formula~\eqref{eqn:model-repn}.  By the implication we have already proved, we have $z \in \cO_Y$.  Since $M(x)$ and $M(z)$ are isomorphic, $M(x)$ also admits a Jordan basis of type $Y$.
\end{proof}

In a Jordan basis, we have $\ker x_i = \spn \{ u_{i1}^{(k)} \mid 1 \le k \le y_{i1} \}$.  More generally, we have
\begin{equation}\label{eqn:kernel-example}
\begin{aligned}
\ker x_{i+j-1} \cdots x_{i+1}x_i &= \spn \{ u_{ih}^{(k)} \mid \text{$1 \le h \le j$,  $1 \le k \le y_{ih}$} \}, \\
\dim \ker x_{i+j-1} \cdots x_{i+1}x_i &= y_{i1} + y_{i2} + \cdots + y_{ij}.
\end{aligned}
\end{equation}

\begin{remark}
A number of basic notions involving quiver representations can be translated into the language of triangular arrays.  We list some examples below, using the following notation:
for $Y \in \bP(\bw)$, we let $M(Y)$ denote the quiver representation corresponding to some point $x \in \cO_Y$.
\begin{enumerate}
\item For $Y \in \bP(\bw)$ and $Z \in \bP(\bv)$, we have $M(Y+Z) \cong M(Y) \oplus M(Z)$.  (Here $Y+Z$ is the entrywise sum of $Y$ and $Z$.)
\item The module $M(Y)$ is an injective object in $\Rep(Q_n)$ if and only if $Y$ is constant along ladders, i.e., if $y_{ij} = y_{i-1,j+1}$ for all $i$ and $j$.
\item The module $M(Y)$ is a projective object in $\Rep(Q_n)$ if and only if $Y$ has nonzero entries only in the last ladder, i.e., if $y_{ij} = 0$ whenever $i+j < n+1$.
\end{enumerate}
\end{remark}

\section{The partial order on orbits}
\label{sec:partialorder}

Let $\le_\bg$ be the partial order on $\bP(\bw)$ induced by the closure order on $G(\bw)$-orbits; that is, for $Y, Y' \in \bP(\bw)$,
\begin{equation}\label{eqn:geom-order-defn}
Y \le_\bg Y' 
\qquad
\text{if $\cO_Y \subset \overline{\cO_{Y'}}$.}
\end{equation}
The goal of this section is to prove that the chutewise dominance order $\le_\bc$ (see \eqref{eqn:comb-order-defn}) and the geometric partial order $\le_\bg$ coincide.

Let $W$ be the symmetric group on $n+1$ letters, i.e., the Weyl group associated to the Dynkin diagram that is the underlying graph of our quiver $Q_n$.  Let $s_i$ (for $i = 1, 2, \ldots, n$) be the transposition that exchanges $i$ and $i+1$.  In other words, these are the simple reflections in $W$. Consider the following reduced expression for the longest element $w_0 \in W$:
\[
w_0 = (s_n)(s_{n-1}s_n)\cdots(s_2s_3 \cdots s_n)(s_1s_2 \cdots s_n).
\]
This reduced expression is ``adapted'' to our quiver in the sense of~\cite[\S4.7]{lusztig}. More precisely, in the notation of~\cite{lusztig}, the sequence
\[
\mathbf{i} = (n,n-1,n,\ldots,2,3,\ldots,n,1,2,\ldots,n) \in \mathscr{H}
\]
is adapted to our quiver.  This sequence determines an ordering on the set of positive roots as in~\cite[\S2.8]{lusztig}.  Denote the positive roots in this order by $\alpha^1, \alpha^2, \ldots, \alpha^{n(n+1)/2}$.  They are given by:
\[
\gamma_{nn}, \gamma_{n-1,n}, \gamma_{n-1,n-1}, \ldots, \gamma_{in}, \gamma_{i,n-1}, \ldots, \gamma_{ii}, \ldots, \gamma_{1n}, \gamma_{1,n-1}, \ldots, \gamma_{11}.
\]
(Recall that $\gamma_{ij} = \be_i + \be_{i+1} + \cdots + \be_j$.)  Note that for $b = (b_{ij}) \in B(\bw)$,  the ordering on the positive roots induces an ordering on the $b_{ij}$.  We write $b^t$ to denote the number $b_{ij}$ corresponding to the positive root $\alpha^t$.

Next, let $\varpi_1^\vee, \ldots, \varpi_n^\vee$ be the fundamental coweights, and let $\phi_{ij} = -\varpi^\vee_{i-1} + \varpi^\vee_j$. Following \cite{b:cbqfm, m:rtgea}, the sequence $\mathbf{i}$ determines a sequence of $n(n+1)/2$ ``chamber coweights'' $\lambda^1, \lambda^2, \ldots, \lambda^{n(n+1)/2}$.  They are given by:
\[
\phi_{nn}, \phi_{n-1,n}, \phi_{n-1,n-1}, \ldots, \phi_{in}, \phi_{i,n-1}, \ldots, \phi_{ii}, \ldots, \phi_{1n}, \phi_{1,n-1}, \ldots, \phi_{11}.
\]
We write $\langle -,- \rangle$ for the usual pairing between coweights and weights. We have the following description of $\le_\bg$.

\begin{theorem}[{\cite[Proposition~4.1 and Remark~4.2(i)]{b:cbqfm}}]
For $Y, Z \in \bP(\bw)$, we have $Y \le_\bg Z$ if and only if
\begin{equation}\label{eqn:comb-order-two-defn}
\sum_{s = 1}^t \langle \lambda^t, \alpha^s \rangle \nu(Y)^s \ge \sum_{s = 1}^t \langle \lambda^t, \alpha^s \rangle \nu(Z)^s
\qquad\text{for all $1\le t \le n(n+1)/2$.}
\end{equation}
\end{theorem}

For $Y,Z \in \bP(\bw)$, let us write $\nu(Y) = (b_{ij})_{1\le i\le j\le n}$ and $\nu(Z) = (c_{ij})_{1 \le i \le j \le n}$.  Consider the following condition:
\begin{equation}\label{eqn:neworder}
\sum_{i=\ell}^k \sum_{j=k}^n b_{ij} \geq \sum_{i=\ell}^k \sum_{j=k}^n c_{ij}
\qquad\text{for all $1 \le \ell \le k \le n$.}
\end{equation}

\begin{lemma}
Let $Y, Z \in \bP(\bw)$.  Then \eqref{eqn:comb-order-two-defn} and \eqref{eqn:neworder} are equivalent conditions.
\end{lemma}
\begin{proof}
Notice that the pairing $\langle \phi_{lk}, \gamma_{ij} \rangle$ appears in the sums from \eqref{eqn:comb-order-two-defn} if and only if $i > \ell$ or $i = \ell$ and $j \ge k$.  Under these conditions,
\[
\langle \phi_{\ell k}, \gamma_{ij} \rangle = \left\{\begin{array}{ll}1 & \text{if } n \ge j \ge k \text{ and } k \ge i \ge \ell \\ 0 & \text{otherwise.} \end{array}\right.
\] 
The claim follows.
\end{proof}

\begin{theorem}\label{thm:po}
The chutewise dominance order $\le_\bc$ on $\bP(\bw)$ coincides with the geometric partial order $\le_{\bg}$.
\end{theorem}

\begin{proof}
Given $Y = (y_{ij})$ and $Z \in (z_{ij})$ in $\bP(\bw)$, write $\nu(Y) = (b_{ij})_{1\le i\le j\le n}$ and $\nu(Z) = (c_{ij})_{1 \le i \le j \le n}$.  Recall from the proof of Lemma~\ref{lem:pb-bij} that
\[
y_{ij} = \sum_{h=1}^i b_{h,i+j-1}
\qquad\text{and}\qquad
z_{ij}=\sum_{h=1}^i c_{h,i+j-1}.
\]

We first observe that for any $Y$ and $Z$ (regardless of how they compare under $\le_\bg$), the $\ell = 1$ case of~\eqref{eqn:neworder} is actually an equality.  Indeed, the two sides simplify to $\sum_{j=k}^n y_{k,j-k+1}$ and $\sum_{j=k}^n z_{k,j-k+1}$, respectively, and both are equal to $w_k$ by the definition of $\bP(\bw)$.

Suppose $1 \le m \le k \le n$.  Here are two (somewhat expanded) instances of the $\ell = 1$ case of~\eqref{eqn:neworder}:
\begin{align}
\sum_{i=1}^{m-1} \sum_{j=k}^n b_{ij} + \sum_{i=m}^{k} \sum_{j=k}^n b_{ij}\label{eqn:ordexp1}
&=
\sum_{i=1}^{m-1} \sum_{j=k}^n c_{ij} + \sum_{i=m}^{k} \sum_{j=k}^n c_{ij}, \\
\sum_{i=1}^{m} \sum_{j=m}^{k} b_{ij} + \sum_{i=1}^{m} \sum_{j=k+1}^n b_{ij}\label{eqn:ordexp2}
&=
\sum_{i=1}^{m} \sum_{j=m}^{k} c_{ij} + \sum_{i=1}^{m} \sum_{j=k+1}^n c_{ij}.
\end{align}
Combining~\eqref{eqn:ordexp1} and~\eqref{eqn:neworder}, we see that $Y \le_\bg Z$ if and only if
\[
\sum_{i=1}^{m-1} \sum_{j=k}^n b_{ij} \leq \sum_{i=1}^{m-1} \sum_{j=k}^n c_{ij}
\qquad\text{for all $1 \le m \le k \le n$,}
\]
or, equivalently,
\begin{equation}\label{eqn:midorder}
\sum_{i=1}^{m} \sum_{j=k+1}^n b_{ij} \leq \sum_{i=1}^{m} \sum_{j=k+1}^n c_{ij}
\qquad\text{for all $1 \le m \le k \le n$.}
\end{equation}
Next,~\eqref{eqn:ordexp2} implies that~\eqref{eqn:midorder} holds if and only if
\[
\sum_{i=1}^{m} \sum_{j=m}^{k} b_{ij} \ge \sum_{i=1}^{m} \sum_{j=m}^{k} c_{ij}
\quad\text{or}\quad
\sum_{j=m}^{k} y_{m,j-m+1} \ge \sum_{j=m-1}^{k-1} z_{m,j-m+1}
\]
for all $1 \le m \le k \le n$.  This is equivalent to~\eqref{eqn:comb-order-defn}, so we conclude that $Y \le_\bg Z$ if and only if $Y \le_\bc Z$.
\end{proof}

\section{Dimensions of orbits}
\label{sec:dimension}

There is an explicit formula for the dimension of any orbit in $E(\bw)$ going back to~\cite[\S 6]{lusztig}, in terms of $B(\bw)$.  (Like the description of $\le_\bg$ given in Section~\ref{sec:partialorder}, the formula requires enumerating the positive roots based on the choice of an adapted reduced expression for $w_0$.)  In this section, we obtain a new dimension formula in terms of $\bP(\bw)$.  Our formula can probably be deduced combinatorially from Lusztig's formula~\cite{lusztig}, but we give a self-contained proof.

\begin{definition}
Let $\bw = (w_1, \ldots, w_n) \in \Z_{\ge 0}^n$, and let $Y \in \bP(\bw)$.  A \emph{kernel flag of type $Y$} is a collection of vector spaces $(V_{ij})_{1 \le i \le n,1 \le j \le n-i+1}$ such that
\[
0 \subset V_{i1} \subset V_{i2} \subset \cdots \subset V_{i,n-i+1} = \C^{w_i}
\qquad\text{and}\qquad
\dim V_{ij} = y_{i1} + y_{i2} + \cdots + y_{ij}.
\]
A quiver representation $x \in E(\bw)$ is said to \emph{preserve the kernel flag $(V_{ij})$} if
\[
x_i(V_{ij}) \subset 
\begin{cases}
0 & \text{if $j = 1$,} \\
V_{i+1,j-1} & \text{if $j > 1$.}
\end{cases}
\]
\end{definition}

This definition implies that if $x$ preserves $(V_{ij})$, then
\begin{equation}\label{eqn:kernel-flag}
V_{ij} \subset \ker x_{i+j-1} \cdots x_{i+1}x_i.
\end{equation}
This observation is the reason for the name ``kernel flag.''  The space of all kernel flags of type $Y$ is denoted by $\Fl_Y$.  Note that $G(\bw)$ acts transitively on $\Fl_Y$.  For any $V \in \Fl_Y$, let $G(\bw)^V$ be its stabilizer in $G(\bw)$.  We then have an isomorphism
\[
\Fl_Y \cong G(\bw)/G(\bw)^V.
\]

Next, for any $V \in \Fl_Y$, let
\[
E(\bw)^V = \{ x \in E(\bw) \mid \text{$x$ preserves the kernel flag $V$} \}.
\]
Then let $\tE_Y$ be the space of pairs
\[
\tE_Y = \{ (V, x) \in \Fl_Y \times E(\bw) \mid x \in E(\bw)^V \}.
\]
This space is a vector bundle over $\Fl_Y$, with fibers isomorphic to $E(\bw)^V$ for any $V \in \Fl_Y$.  In particular, $\tE_Y$ is a smooth, irreducible variety.  We denote by
\[
\pi_Y: \tE_Y \to E(\bw)
\]
the projection map onto the second factor.  This map is proper.  Finally, for another description of $\tE_Y$, choose a point $V \in \Fl_Y$. Then there is an isomorphism
\[
G(\bw) \times^{G(\bw)^V} E(\bw)^V \overset{\sim}{\to} \tE_Y
\]
given by $(g,x) \mapsto (gV,g \cdot x)$.  

\begin{lemma}
Let $Y \in \bP(\bw)$, and let $V \in \Fl_Y$.  Then we have
\[
\dim \Fl_Y = \hspace{-1em}\sum_{\substack{1 \le i \le n-1\\ 1 \le j < k \le n-i+1}}\hspace{-1em} y_{ij}y_{ik}
\qquad\text{and}\qquad
\dim E(\bw)^V = \hspace{-1em}\sum_{\substack{1 \le i \le n-1\\ 1 \le j < k \le n-i+1}}\hspace{-1em} y_{i+1,j}y_{ik}.
\]
\end{lemma}
\begin{proof}
Let us first compute $\dim \Fl_Y$. We begin by recalling that
\[
\dim \GL(w_i) = w_i^2 = (y_{i1} + \cdots + y_{i,n-i+1})^2 = \sum_{1 \le j,k \le n-i+1} y_{ij}y_{ik}.
\]
Consider the point $V = (V_{ij}) \in \Fl_Y$.  For each $i$, let $\GL(w_i)^{V_{i\bullet}}$ denote the stabilizer of the partial flag $0 \subset V_{i1} \subset V_{i2} \subset \cdots \subset V_{i,n-i+1} = \C^{w_i}$.  Then $G(\bw)^V$ is the product of the various $\GL(w_i)^{V_{i\bullet}}$.  Let us compute the dimension of the latter.  Choose a splitting of the flag, i.e., a vector space isomorphism
\[
\C^{w_i} = V_{i1} \oplus (V_{i2}/V_{i1}) \oplus \cdots \oplus (V_{i,n-i+1}/V_{i,n-i}).
\]
Note that $\dim V_{ij}/V_{i,j-1} = y_{ij}$.  We have
\[
\GL(w_i)^{V_{i\bullet}} \cong \prod_{j=1}^{n-i+1} \GL(V_{ij}/V_{i,j-1}) \ltimes \prod_{1 \le k < j \le n-i+1} \Hom(V_{ik}/V_{i,k-1}, V_{ij}/V_{i,j-1}),
\]
where we use the convention that if $j= 1$, then $V_{i,j-1} = 0$. Therefore,
\[
\dim \GL(w_i)^{V_{i\bullet}} = \sum_{j=1}^{n-i+1} y_{ij}^2 +
\sum_{1 \le k < j \le n-i+1} y_{ij}y_{ik}
= \sum_{1 \le k \le j \le n-i+1} y_{ij}y_{ik}.
\]
We are now ready to compute the dimension of $\Fl_Y$.  We have
\begin{multline*}
\dim \Fl_Y = \dim G(\bw) - \dim G(\bw)^V
= \sum_{i=1}^n (\dim \GL(\bw_i) - \dim \GL(\bw_i)^{V_{i\bullet}}) \\
= \sum_{i=1}^n \left(\sum_{1 \le j,k \le n-i+1} y_{ij}y_{ik} -  \sum_{1 \le k \le j \le n-i+1} y_{ij}y_{ik} \right)
= \sum_{i=1}^n \sum_{1 \le j < k \le n-i+1} y_{ij}y_{ik},
\end{multline*}
as desired.

Next, for $x = (x_i) \in E(\bw)^V$, we must have
\[
x_i \in \prod_{k=2}^{n-i+1} \Hom(V_{ik}/V_{i,k-1}, V_{i+1,k-1}).
\]
The dimension of the space on the right-hand side above is
\[
\sum_{k=2}^{n-i+1} y_{ik}(y_{i+1,1} + y_{i+1,2} + \cdots + y_{i+1,k-1})
= \sum_{1 \le j < k \le n-i+1} y_{i+1,j}y_{ik}.
\]
The dimension of $E(\bw)^V$ is the sum of these quantities over all $i$.
\end{proof}

\begin{lemma}
\label{lem:tE-open}
There is an open subset $U \subset \tE_Y$ such that $\pi_Y$ restricts to a bijection $U \to \cO_Y$.
\end{lemma}
\begin{proof}
Choose a point $V = (V_{ij}) \in \Fl_Y$.  Let $U_V \subset E(\bw)^V$ be the subset consisting of elements $x \in E(\bw)^V$ such that when $j > 1$, the map of quotient spaces
\[
V_{ij}/V_{i,j-1} \to V_{i+1,j-1}/V_{i+1,j-2}
\]
induced by $x_i$ is injective.  Note that $U_V$ is an open subset: with an appropriate choice of bases, the injectivity of these induced maps is equivalent to the nonvanishing of certain minors of the matrix for $x_i$.  The quotient map $q: G(\bw) \times E(\bw)^V \to G(\bw) \times^{G(\bw)^V} E(\bw)^V \cong \tE_Y$ is an open map, so the set $U = q(G(\bw) \times U_V)$ is open.

Let $x \in \cO_Y$.  We will show that $\pi_Y^{-1}(x)$ consists of a single point, and that that point lies in $U$.  Choose a Jordan basis $\{ u_{ij}^{(k)}\}$ for $M(x)$.  Comparing~\eqref{eqn:kernel-example} with~\eqref{eqn:kernel-flag}, we see that there is a unique kernel flag of type $Y$ preserved by $x$: namely,
\[
V_{ij} = \ker x_{i+j-1} \cdots x_{i+1}x_i.
\]
In other words, $\pi_Y^{-1}(x)$ consists of a single point.  The quotient space $V_{ij}/V_{i,j-1}$ can then be identified with the span of $\{ u_{ij}^{(k)} \mid 1 \le k \le y_{ij} \}$, so~\eqref{eqn:model-repn} shows us that the induced map $V_{ij}/V_{i,j-1} \to V_{i+1,j-1}/V_{i+1,j-2}$ is injective.  Thus, the point $((V_{ij}),x)$ belongs to $U$.

For the opposite direction, we start with a point $((V_{ij}),x) \in U$.  We will prove that $x \in \cO_Y$.  
We will construct a certain basis  $\{ u_{ij}^{(k)} \}_{1 \le j \le n-i+1, 1 \le k \le y_{ij}}$ for $\C^{w_i}$ with the property that for any $m \le n-i+1$, 
\begin{equation}\label{eqn:tE-open-basis}
\text{$\{ u_{ij}^{(k)} \mid 1 \le j \le m,\ 1 \le k \le y_{ij} \}$ is a basis for $V_{im} \subset \C^{w_i}$.}
\end{equation}
We proceed by induction on $i$.  For $i = 1$, choose any basis $\{u_{1j}^{(k)}\}$ satisfying~\eqref{eqn:tE-open-basis}.  For $i > 1$, define 
\[
u_{ij}^{(k)} = x(u_{i-1,j+1}^{(k)}) \qquad\text{if $1 \le k \le y_{i-1,j+1}$.}
\]
Since the map $V_{i-1,j+1}/V_{i-1,j} \to V_{ij}/V_{i,j-1}$ induced by $x$ is injective, these elements are linearly independent.  Therefore, it is possible to find additional elements $\{ u_{ij}^{(k)} \}_{1 \le j \le n-i+1, y_{i-1,j+1} < k \le y_{ij}}$ so that the whole collection forms a basis for $\C^{w_i}$ satisfying~\eqref{eqn:tE-open-basis}.  Since $x(V_{i1}) = 0$, we have $x(u_{i1}^{(k)}) = 0$ for all $i$ and $k$.  Thus, our basis satisfies~\eqref{eqn:model-repn}, and we conclude that $x \in \cO_Y$.
\end{proof}

\begin{corollary}
\label{cor:tE-dim}
We have $\dim \tE_Y = \dim \cO_Y$.
\end{corollary}

\begin{corollary}
\label{cor:tE-image}
The image of $\pi_Y: \tE_Y \to E(\bw)$ is $\overline{\cO_Y}$.
\end{corollary}
\begin{proof}
Since $\tE_Y$ is irreducible, it is the closure of the open set $U$ that was introduced in Lemma~\ref{lem:tE-open}.  Its image must therefore be contained in the closure of $\pi_Y(\tE_Y) = \cO_Y$.  Since $\pi_Y$ is proper, its image is closed, so its image is precisely $\overline{\cO_Y}$.
\end{proof}

Combining the preceding results, we obtain the following dimension formula.

\begin{theorem}\label{thm:dim}
For any $Y \in \bP(\bw)$, we have
\[
\dim \cO_Y =
\hspace{-1em}\sum_{\substack{1 \le i \le n-1\\ 1 \le j < k \le n-i+1}}\hspace{-1em} y_{ij}y_{ik} + \hspace{-1em}\sum_{\substack{1 \le i \le n-1\\ 1 \le j < k \le n-i+1}}\hspace{-1em} y_{i+1,j}y_{ik}.
\]
\end{theorem}

\section{Operations on triangular arrays}
\label{sec:operations}

This section is the ``combinatorial heart'' of the paper.  We describe a number of constructions one can carry out using triangular arrays, culminating in the definitions of two maps $\sT, \sT': \bP(\bw) \to \bP(\bw^*)$.  The main result of this section states that $\sT$ and $\sT'$ are both bijections, inverse to one another.  (In Section~\ref{sec:geometry}, we will learn that $\sT$ and $\sT'$ are actually the same map, but the proof of this is not combinatorial.)

\subsection{Elementary operations on triangular arrays}

Consider a triangular array $Y = (y_{ij})_{1 \le i \le n, 1 \le j \le n-i+1}$ of size $n$.  
We define $\Delc(Y)$ to be the triangular array of size $n-1$ obtained from $Y$ by deleting the first chute.  In other words,
\[
\Delc(Y)_{ij} = y_{i+1,j}
\qquad\text{for $1 \le i \le n-1$, $1 \le j \le n-i$.}
\]
Similarly, $\Dell(Y)$ is the triangular array of size $n-1$ obtained by deleting the last ladder:
\[
\Dell(Y)_{ij} = y_{ij}
\qquad\text{for $1 \le i \le n-1$, $1 \le j \le n-i$.}
\]

On the other hand, let $Q = (q_1, \ldots, q_{n+1})$ be a list of $n+1$ nonnegative integers.  Assume first that $q_j \ge y_{1,j-1}$ for $2 \le j \le n+1$.  Let $Y \cupc Q$ be the triangular array of size $n+1$ obtained from $Y$ by making $Q$ the new topmost chute.  In other words,
\[
(Y \cupc Q)_{ij} =
\begin{cases}
q_j & \text{if $i = 1$, $1 \le j \le n+1$,} \\
y_{i-1,j} & \text{if $2 \le i \le n+1$, $1 \le j \le n-i+2$.}
\end{cases}
\]
Similarly, if we instead assume that $q_1 \ge q_2 \ge \cdots \ge q_{n+1}$, then we can define a new triangular array $Y \cupl Q$ be adjoining $Q$ as the new bottommost ladder.  Explicitly, we put
\[
(Y \cupl Q)_{ij} =
\begin{cases}
y_{ij} & \text{if $1 \le i \le n$ and $1 \le j \le n-i+1$,} \\
q_{n-i+2} & \text{if $1 \le i \le n+1$ and $j = n-i+2$.}
\end{cases}
\]

Let $\Top(Y)$ denote the topmost chute of $Y$, regarded as an element of $\Z^n$:
\[
\Top(Y) = (y_{11}, y_{12}, \ldots, y_{1n}).
\]
Note that
\[
Y = \Delc(Y) \cupc \Top(Y).
\]

Next, we define $\Raise(Y,i,j)$ and $\Lower(Y,i,j)$ to be the triangular arrays obtained from $Y$ by replacing the entry in chute $i$, column $j$ by $y_{ij}+1$ and by $y_{ij}-1$, respectively.  There is a well-definedness issue here: because ladders are required to be weakly decreasing, $\Raise(Y,i,j)$ only makes sense if $j = 1$ or if $y_{ij} < y_{i+1,j-1}$.  Similarly, for $\Lower(Y,i,j)$ to make sense, we must either have $i = 1$ and $y_{1j} > 0$, or else $i > 1$ and $y_{ij} > y_{i-1,j+1}$.  When they make sense, it is clear from the definitions that
\[
\udim(\Raise(Y,i,j)) = \udim(Y) + \be_i
\qquad\text{and}\qquad
\udim(\Lower(Y,i,j)) = \udim(Y) - \be_i.
\]

\subsection{Invariants of triangular arrays}

In this subsection, we define various in\-te\-ger-valued functions on triangular arrays that will be used in the definitions of the algorithms below.  As above, let $Y$ be a triangular array of size $n$.  Let $k$ be an integer with $1 \le k \le n$.  Let
\[
\cI(Y,k) =
\begin{cases}
\text{the smallest integer $j \ge k$ such that $y_{1j} > 0$, or} \\
\text{$\infty$, if there is no such $j$.}
\end{cases}
\]
Next, let
\[
\cJ(Y,k) =
\begin{cases}
\text{the smallest integer $j > \cI(Y,k)$ such that $y_{1j} < y_{2,j-1}$,} \\
\text{\qquad if $1 < \cI(Y,k) < \infty$, or} \\
\text{$\infty$, if there is no $j$ as in the previous case, or if $\cI(Y,k) = \infty$.}
\end{cases}
\]
In other words, if $\cJ(Y,k) < \infty$, then it is the smallest integer${}\ge \cI(Y,k)$ such that $\Raise(Y, 1, \cJ(Y,k))$ is defined.  In particular, we always have $\cJ(Y,k) > 1$.

Finally, suppose $1 \le i \le n-k+1$.  Let
\[
\cK_i(Y,k) =
\max \big(\{1\} \cup \{j \mid \text{$2 \le j \le k$ and $y_{ij} < y_{i+1,j-1}$}\}\big).
\]
In other words, $\cK_i(Y,k)$ is the largest integer${}\le k$ such that $\Raise(Y,i,\cK_i(Y,k))$ is defined.  It is immediate from the definition that if $2 \le i \le n-k+1$, then we have
\begin{equation}\label{eqn:ckdel}
\cK_i(Y,k) = \cK_{i-1}(\Delc(Y),k).
\end{equation}

\subsection{Advanced operations on triangular arrays}

We will now introduce several more complicated operations on triangular arrays, and we prove a few lemmas about them.

\subsubsection*{Procedure $\sa$}  This operation takes as input a triple $(Y,i,k)$ where $Y$ is a triangular array of size $n$; $i$ is an integer such that $1 \le i \le n$; and $k$ is an integer such that $1 \le k \le n-i+1$.  Its output is also a triple consisting of a triangular array and two integers.  It is defined by
\[
\sa(Y,i,k) = (\Raise(Y,i, \cK_i(Y,k)), i-1, \cK_i(Y,k)).
\]
Note that as long as $i > 1$, the output of $\sa$ satisfies the conditions required of its input, so it makes sense to apply $\sa$ repeatedly.  

When $i > 1$, we can study how $\sa$ interacts with $\Delc$ using~\eqref{eqn:ckdel}.  Suppose
\[
\sa(Y,i,k) = (X,i-1,k')
\qquad\text{and}\qquad
\sa(\Delc(Y),i-1,k) = (X',i-2,k'').
\]
These are related by
\begin{equation}\label{eqn:bdelc}
X = X' \cupc \Top(Y)
\qquad\text{and}\qquad k' = k''.
\end{equation}

\subsubsection*{Procedure $\sA_i$} This operation takes as input a triangular array $Y$ of size $n$, where $1 \le i \le n$.  Its output it also a triangular array of size $n$.  Apply procedure $\sa$ $i$ times to the triple $(Y,i,n)$: the result has the form
\[
\underbrace{\sa \circ \cdots \circ \sa}_{\text{$i$ times}}(Y,i,n-i+1) = (X,0,k).
\]
We define $\sA_i(Y) = X$.  Since this sequence of $\sa$'s performs one $\Raise$ on each of the first $i$ chutes, we see that
\begin{equation}\label{eqn:dimtau}
\udim(\sA_i(Y)) = \udim(Y) + \be_1 + \be_2 + \cdots + \be_i.
\end{equation}

\subsubsection*{Procedure $\sB$}  This operation takes as input a pair $(Y,k)$, where $Y$ is a triangular array of size $n$; $k$ is an integer such that $1 \le k \le n$; and, moreover, we have $\cI(Y,k) < \infty$.  Its output is again a pair consisting of a triangular array and an integer (not necessarily satisfying any condition with respect to $\cI$). The definition is by induction on $n$.  If $n = 1$, we necessarily have $k = 1$.  In this case, we put
\[
\sB(Y,1) = (\Lower(Y,1,1),1).
\]
(The assumption that $\cI(Y,1) < \infty$ implies that this use of $\Lower$ makes sense.)

Suppose now that $n > 1$, and that $\sB$ is already defined for smaller diagrams.  If $\cJ(Y,k) = \infty$, we simply put
\[
\sB(Y,k) =
(\Lower(Y,1, \cI(Y,k)), 1).
\]
On the other hand, if $\cJ(Y,k) < \infty$, let $j_0 = \cJ(Y,k)$.  Our assumption implies that $\Raise(Y,1,j_0)$ makes sense, so $y_{1j_0} < y_{2,j_0-1}$.  In particular, $y_{2,j_0-1} \ne 0$, and hence $\cI(\Delc(Y), j_0-1) < \infty$.  By induction, $\sB(\Delc(Y), j_0-1)$ is already defined; let $(Z,r) = \sB(\Delc(Y), j_0 - 1)$.  Finally, set
\[
\sB(Y,k) = (\Lower(Z \cupc \Top(Y), 1, \cI(Y,k)), r+1).
\]
This completes the definition of $\sB$.  Note that the definition for $n = 1$ is a special case of the definition in the case where $\cJ(Y,k) = \infty$.

\begin{lemma}\label{lem:cdim}
Suppose that $\cI(Y,k) < \infty$, and let $(Y', q) = \sB(Y,k)$.  Then $\udim(Y') = \udim(Y) - (\be_1 + \be_2 + \cdots + \be_q)$.
\end{lemma}
\begin{proof}
We proceed by induction on $n$.  If $n = 1$, or if $\cJ(Y,k) = \infty$, we have $q = 1$, and $Y'$ is given by a $\Lower$. The claim is clear in this case.

If $\cJ(Y,k) < \infty$, suppose $\udim(Y) = (w_1, \ldots, w_n)$. Then $\udim(\Delc(Y)) = (w_2, \ldots, w_n)$.  Let $(Z,r)$ be as in the definition of $\sB$ above. By induction, $\dim(Z) = (w_2 -1, w_3 -1, \ldots, w_q-1, w_{q+1}, \ldots, w_n)$.  Then $Y' = \Lower(Z \cupc \Top(Y), 1, \cI(Y,k))$ has the dimension vector claimed in the lemma.
\end{proof}

\begin{lemma}\label{lem:qorder}
Let $(Y', q_1) = \sB(Y, k)$, and let $(Y'', q_2) = \sB(Y', k')$ for some $k' \geq k$.  Then $q_1 \geq q_2$.
\end{lemma}
\begin{proof}
If $\cJ(Y,k) = \infty$, then $q_1 = q_2 = 1$, and the lemma is verified.  Now assume that $\cJ(Y,k)$ exists.  Certainly 
\begin{equation}\label{eqn:onestep}
(\Delc(Y'), q_1 - 1) = \sB(\Delc(Y), \cJ(Y, k) - 1).
\end{equation}
Applying this to $(Y', k')$ yields 
\begin{equation}\label{eqn:twostep}
(\Delc(Y''), q_2 - 1) = \sB(\Delc(Y'), \cJ(Y', k') - 1).
\end{equation}
Note that $\cJ(Y', k') \geq \cJ(Y, k)$, so that~\eqref{eqn:onestep} and~\eqref{eqn:twostep} match the statement of the lemma for smaller triangles.  By induction, $q_1 - 1 \geq q_2 - 1$, and we are done.
\end{proof}

\begin{lemma}\label{lem:singletau}
Suppose that $\cI(Y,k) < \infty$, and let $(Y', q) = \sB(Y,k)$.  Then we have $\sa^q(Y',q,n-q+1) = (Y, 0, \cI(Y,k))$.  In particular, we have $\sA_q(Y') = Y$.
\end{lemma}
\begin{proof}
We proceed by induction on the size $n$ of the triangular array.  Throughout the proof, we let $j_0 = \cI(Y,k)$.

Suppose first that $\cJ(Y,k) = \infty$.  (This includes the special case where $n = 1$.)  Recall that $j_0$ is the smallest integer${}\ge k$ such that $y_{1,j_0} \ne 0$.  Moreover, if $j > j_0$, then $\Raise(Y,1,j)$ is not defined.  From the definition of $\sB$, we have $q = 1$ and $Y' = \Lower(Y,1,j_0)$.  Since $Y$ and $Y'$ differ only at the entries at position $ij$, we see that $\Raise(Y',1,j)$ is also not defined for $j > j_0$.  On the other hand, $\Raise(Y',1,j_0)$ clearly is defined: it is equal to $Y$.  We have just shown that $\cK_1(Y',n) = j_0$.  As a consequence, we have
\begin{multline*}
\sa(Y',1,n) = (\Raise(Y',1,\cK_1(Y',n)),0,\cK_1(Y',n))\\
 = (\Raise(Y',1,j_0),0,j_0) = (Y,0,\cI(Y,k)),
\end{multline*}
as desired.

Now suppose that $\cJ(Y,k) < \infty$, and let $j_1 = \cJ(Y,k)$. From the definition of $\sB$, we see that $\sB(\Delc(Y), j_1-1)$ is of the form $(Z,q-1)$ for some triangular array $Z$ of size $n-1$.  By induction, we have
\[
\sa^{q-1}(Z,q-1,n-q+1) = (\Delc(Y),0,\cI(\Delc(Y),j_1-1)).
\]
Recall from the definition of $\sB$ that
\begin{equation}\label{eqn:singletau}
Y' = \Lower(Z \cupc \Top(Y),1,j_0).
\end{equation}
In particular, we have $\Delc(Y') = Z$.  Applying~\eqref{eqn:bdelc} $q-1$ times, we obtain
\[
\sa^{q-1}(Y',q,n-q+1) = (\Delc(Y) \cupc \Top(Y'),1, \cI(\Delc(Y),j_1-1)).
\]
To finish the proof of the lemma, we must show that if we apply $\sa$ one more time to this equation, the result is $(Y,0,j_0)$.  Let $Y'' = \Delc(Y) \cupc \Top(Y')$.  It follows from~\eqref{eqn:singletau} that $Y = \Raise(Y'', 1, j_0)$, so to complete the proof, it is enough to show that $\cK_1(Y'',j_1-1) = j_0$.  Denote the entries of $Y''$ by $y''_{ij}$.  We have
\begin{gather*}
y''_{1,j_0} = y_{1,j_0} - 1 < y_{2,j_0-1} = y''_{2,j_0-1}, \\
y''_{1,j} = y_{1,j} = y_{2,j-1} = y''_{2,j-1} \qquad \text{for $j_0 < j \le j_1-1$,}
\end{gather*}
where the latter holds by the definition of $\cJ(Y,k)$.  These two conditions together tell us that $\cK_1(Y'',j_1-1) = j_0$, as desired.
\end{proof}

\subsection{The combinatorial Fourier transform and its inverse}
\label{ss:cft}

We are now ready to define the main combinatorial algorithms in the paper.  Let $Y$ be a triangular array of size $n$.  We will define the \emph{combinatorial Fourier transform of $Y$}, denoted by $\sT(Y)$, by induction on $n$.  If $n = 1$, we set $\sT(Y) = Y$.  Otherwise, we set
\[
\sT(Y) = \sA_n^{y_{1,n}} \sA_{n-1}^{y_{2,n-1} - y_{1,n}} \cdots \sA_1^{y_{n,1} - y_{n-1,2}}(\sT(\Dell(Y)) \cupc (0,\ldots, 0))
\]
Note that $\sT(Y)$ is again a triangular array of size $n$.

We will also define the \emph{inverse combinatorial Fourier transform of $Y$}, denoted by $\sICFT(Y)$, by induction on $n$.  If $n = 1$, we again set $\sICFT(Y) = Y$.  Suppose now that $n > 1$, and let $(w_1, \ldots, w_n) = \udim(Y)$.  Recall that $\sB(Y,1)$ is defined as long as the top chute of $Y$ is nonzero, or equivalently, if $w_1 > 0$.  If this is the case, set $(Y',q_1) = \sB(Y,1)$.  Recall from Lemma~\ref{lem:cdim} that the first coordinate of $\udim(Y')$ is $w_1 - 1$.  If this is still positive, we can apply $\sB$ again.  In general, we produce a sequence as follows:
\begin{align*}
(Y',q_1) &= \sB(Y,1), \\
(Y'',q_2) &= \sB(Y',1), \\
&\vdots \\
(Y^{(w_1)},q_{w_1}) &= \sB(Y^{(w_1-1)},1).
\end{align*}
(The top chute of $Y^{(w_1)}$ is zero, so we cannot apply $\sB$ again.)  Define a list of integers $P = (p_1, \ldots, p_n)$ by
\[
p_j = \#\{ k \mid q_k \ge j \}.
\]
(Since $1 \le q_k \le n$ for all $k$, we have $p_1 = w_1$.)
Finally, we define $\sICFT(Y)$ by
\[
\sICFT(Y) = \sICFT(\Delc(Y^{(w_1)})) \cupl P.
\]
The terminology is justified by Theorem~\ref{thm:cftbij} below.

\begin{lemma}
If $\udim(Y) = \bw$, then $\udim(\sT(Y)) = \udim(\sICFT(Y)) = \bw^*$.
\end{lemma}
\begin{proof}
Let us first prove the statement for $\sT$.  We proceed by induction on the size $n$ of the triangular array involved.  If $n = 1$, the statement is obvious.  Otherwise, suppose $\bw = (w_1,\ldots,w_n)$, and let $\bw' = \udim(\Dell(Y))$.  Writing $\bw' = (w'_1, \ldots, w'_{n-1})$, we clearly have
\[
w'_i = w_i - y_{i,n-i+1}.
\]
By induction, $\udim(\sT(\Dell(Y))) = (\bw')^*$, so
\[
\udim(\sT(\Dell(Y) \cupc (0,\ldots,0)) = (0, w_{n-1} - y_{n-1,2}, \ldots, w_2 - y_{2,n-1}, w_1 - y_{1,n}).
\]
Next,~\eqref{eqn:dimtau} implies that
\begin{align*}
\udim \sA_n^{y_{1,n}} &\sA_{n-1}^{y_{2,n-1} - y_{1,n}} \cdots \sA_1^{y_{n,1} - y_{n-1,2}}(\sT(\Dell(Y)) \cupc (0,\ldots, 0)) \\
&= \udim(\sT(\Dell(Y)) \cupc (0,\ldots, 0)) \\
&\qquad + \sum_{i=1}^n (y_{i,n-i+1} - y_{i-1,n-i+2})(\be_1 + \be_2 + \cdots + \be_{n-i+1}).
\end{align*}
(On the right-hand side, $y_{0,n+1}$ should be understood to be $0$.)  The coefficient of $\be_k$ is $\sum_{i=1}^{n-k+1} (y_{i,n-i+1} - y_{i-1,n-i+2}) = y_{n-k+1,k}$.  Using the fact that $y_{n,1} = w_n$, we conclude that
\begin{align*}
\udim \sT(Y) &= \udim(\sT(\Dell(Y)) \cupc (0,\ldots, 0)) + \sum_{k=1}^n y_{n-k+1,k}\be_k \\
&= \udim(\sT(\Dell(Y)) \cupc (0,\ldots, 0))  + (y_{n,1}, y_{n-1,2}, \ldots, y_{1,n}) \\
&= (w_n, w_{n-1}, \ldots, w_1) = \bw^*.
\end{align*}

For $\sICFT$, we again proceed by induction on $n$.  Consider the triangular arrays $Y', Y'', \ldots, Y^{(w_1)}$ as in the definition of $\sICFT$.  Applying Lemma~\ref{lem:cdim} $w_1$ times, we see that
\[
\udim(Y^{(w_1)}) = \udim(Y) - \sum_{k=1}^{w_1} (\be_1 + \cdots + \be_{q_k}).
\]
The coefficient of $\be_j$ is $\#\{ k \mid q_k \ge  j \} = p_j$, so
\[
\udim(Y^{(w_1)}) = \bw - P = (0, w_2 - p_2, w_3 - p_3, \ldots, w_n - p_n).
\]
By induction, we have $\udim (\sICFT(\Delc(Y^{(w_1)}))) = (w_n-p_n, w_{n-1} - p_{n-1}, \ldots, w_2 - p_2)$, and then
\begin{multline*}
\udim (\sICFT(Y))
= \udim (\sICFT(\Delc(Y^{(w_1)}))) + (p_n,p_{n-1}, \ldots, p_1)\\
 = (w_n, w_{n-1}, \ldots, w_1) = \bw^*,
\end{multline*}
as desired.
\end{proof}

The previous lemma tells us that both $\sT$ and $\sICFT$ can be regarded as maps from $\bP(\bw)$ to $\bP(\bw^*)$, or vice versa.

\begin{theorem}\label{thm:cftbij}
Let $\bw \in \Z_{\ge0}^n$.  The maps $\sT: \bP(\bw) \to \bP(\bw^*)$ and $\sICFT: \bP(\bw^*) \to \bP(\bw)$ are both bijections, and they are inverse to one another.
\end{theorem}
\begin{proof}
We begin by showing that $\sT \circ \sICFT$ is the identity map on $\bP(\bw^*)$.  We proceed by induction on $n$. If $n = 1$, the claim is obvious.  Otherwise, let $\bw = (w_1, \ldots, w_n)$.  Let $Y \in \bP(\bw^*)$, and let $Y', Y'', \ldots, Y^{(w_n)}$ be as in the definition of $\sICFT$.  By Lemma~\ref{lem:singletau}, we have
\[
Y = \sA_{q_1} \sA_{q_2} \cdots \sA_{q_{w_n}}(Y^{(w_n)}).
\]
Next, Lemma~\ref{lem:qorder} tells us that $q_1 \ge q_2 \ge \cdots \ge q_{w_n}$.  So the preceding equation can be rewritten as 
\begin{equation}\label{eqn:cftbij1}
Y = \sA_n^{p_n} \sA_{n-1}^{p_{n-1} - p_n} \cdots \sA_2^{p_2 - p_3} \sA_1^{p_1 - p_2}(Y^{(w_n)}).
\end{equation}
Now, the top chute of $Y^{(w_n)}$ is zero, so $Y^{(w_n)} = \Delc(Y^{(w_n)}) \cupc (0,\ldots,0)$.  Since $\Delc(Y^{(w_n)})$ is a triangular array of smaller size, by induction, we have
\[
\Delc(Y^{(w_n)}) = \sT(\sICFT(\Delc(Y^{(w_n)}))).
\]
Next, from the definition of $\sICFT$, we see that
\[
\sICFT(\Delc(Y^{(w_n)})) = \Dell(\sICFT(Y)).
\]
Combining these observations, we find that
\begin{multline*}
Y^{(w_n)} = \Delc(Y^{(w_n)}) \cupc (0,\ldots,0)
= \sT(\sICFT(\Delc(Y^{(w_n)}))) \cupc (0, \ldots, 0)\\
= \sT(\Dell(\sICFT(Y))) \cupc (0, \ldots 0).
\end{multline*}
Finally, we substitute this into~\eqref{eqn:cftbij1} to obtain
\[
Y = \sA_n^{p_n} \sA_{n-1}^{p_{n-1} - p_n} \cdots \sA_2^{p_2 - p_3} \sA_1^{p_1 - p_2}(\sT(\Dell(\sICFT(Y))) \cupc (0, \ldots 0)).
\]
Since the numbers $p_1, p_2, \ldots, p_n$ are precisely those on the bottom ladder of $\sICFT(Y)$, this formula says that $Y = \sT(\sICFT(Y))$, as desired.

We now know that $\sT \circ \sICFT$ is the identity map.  In particular, $\sICFT$ is injective, and $\sT$ is surjective.  Since the finite sets $\bP(\bw)$ and $\bP(\bw^*)$ have the same cardinality (Corollary~\ref{cor:samecard}), we conclude that $\sT$ and $\sICFT$ are both bijections.
\end{proof}

\section{Fourier--Sato transforms}
\label{sec:geometry}

In this section, we prove the main result of the paper: both $\sT$ and $\sICFT$ compute the Fourier--Sato transforms of simple perverse sheaves on $E(\bw)$.

\subsection{Fourier--Sato transform}

Let $\bw \in \Z_{\ge 0}^n$.  There is an obvious isomorphism $G(\bw) \cong G(\bw^*)$, given by $(g_1, g_2, \ldots, g_n) \mapsto (g_n, g_{n-1}, \ldots, g_1)$.  In this section, we will identify these groups via this isomorphism.

Consider the pairing $\langle{-},{-}\rangle: E(\bw) \times E(\bw^*) \to \C$ defined as follows: for $x = (x_i)_{1 \le i \le n-1} \in E(\bw)$ and $y = (y_i)_{1\le i \le n-1} \in E(\bw^*)$, we put
\[
\langle x, y \rangle = \sum_{i=1}^{n-1} \tr(x_i y_{n-i}).
\]
This pairing is $G(\bw)$-equivariant and nondegenerate, so it identifies $E(\bw^*)$ with the dual vector space to $E(\bw)$ (as a $G(\bw)$-representation).  Following~\cite[\S3.7]{ks}, one can define the \emph{Fourier--Sato transform}, a certain functor $\Db_{G(\bw)}(E(\bw)) \to \Db_{G(\bw^*)}(E(\bw^*))$ denoted in~\cite{ks} by $\cF \mapsto \cF^\wedge$.  In this paper, it will be more convenient to use the functor
\[
\bT: \Db_{G(\bw)}(E(\bw)) \to \Db_{G(\bw^*)}(E(\bw^*))
\]
defined by $\bT(\cF) = (\cF^\wedge)[\dim E(\bw)]$.  With this additional shift, $\bT$ becomes $t$-exact for the perverse $t$-structure.  It is an equivalence of categories (because $G(\bw)$-equivariant sheaves are automatically conic in the sense of~\cite[\S3.7]{ks}), and it is ``almost'' an involution: its inverse
\[
\bT': \Db_{G(\bw^*)}(E(\bw^*)) \to \Db_{G(\bw)}(E(\bw))
\]
is given by
\[
\bT'(\cF) \cong s^*\bT(\cF),
\]
where $s: E(\bw) \to E(\bw)$ is the ``antipode map'' given by $s(x) = -x$.

\subsection{Simple perverse sheaves on \texorpdfstring{$E(\bw)$}{E(w)}}

The following fact is well known.  We include the proof since it is so short.

\begin{lemma}\label{lem:stab-conn}
For any point $x \in E(\bw)$, the stabilizer of $x$ in $G(\bw)$ is connected.
\end{lemma}
\begin{proof}
Let $\mathfrak{g}(\bw) = \mathfrak{gl}(w_1) \times \cdots \times \mathfrak{gl}(w_n)$, and let this Lie algebra act on $E(\bw)$ by
\[
(g_1,\ldots,g_n) \cdot (x_1,\ldots,x_{n-1}) = (g_2x_1, g_3x_2, \ldots, g_nx_{n-1}) - (x_1g_1, x_2g_2, \ldots , x_{n-1}g_{n-1}).
\]
Let $\mathfrak{z}$ be the stabilizer in $\mathfrak{g}(\bw)$ of a point $x \in E(\bw)$.  Then $\mathfrak{z}$ is a vector space.  The stabilizer in $G(\bw)$ of $x$ is the Zariski open subset of $\mathfrak{z}$ consisting of elements with nonzero determinant, so it is connected.
\end{proof}

As a consequence, the only $G(\bw)$-equivariant irreducible local system on any $G(\bw)$-orbit is the trivial local system.  Every simple $G(\bw)$-equivariant perverse sheaf is therefore of the form $\IC(\cO_Y)$ for some $Y \in \bP(\bw)$.  Given such a perverse sheaf, its Fourier--Sato transform $\bT(\IC(\cO_Y))$ is a $G(\bw)$-equivariant simple perverse sheaf on $E(\bw^*)$, so it must be isomorphic to $\IC(\cO_{Y'})$ for some $Y' \in \bP(\bw)$.  We thus obtain a map
\[
\bT: \bP(\bw) \to \bP(\bw^*)
\]
characterized by the property that $\bT(\IC(\cO_Y)) \cong \IC(\cO_{\bT(Y)})$.  This is a bijection.  Note that the antipode map $s: E(\bw) \to E(\bw)$ preserves every $G(\bw)$-orbit.  It follows that at the combinatorial level, $\bT$ is an involution:
\begin{equation}\label{eqn:bt-inv}
\bT(\bT(Y)) = Y \qquad\text{for all $Y \in \bP(\bw)$.}
\end{equation}

\begin{lemma}\label{lem:ses-raise}
Suppose we have a short exact sequence of representations $0 \to M(x) \to M(x') \to M(\be_i) \to 0$.  Let $v$ be a vector in $M(x')$ whose image in $M(\be_i)$ is nonzero, and let $k$ be the smallest integer such that $(x')^k(v) = 0$.  If $x \in \cO_Y$ and $x' \in \cO_{Y'}$, then
\[
Y' = \Raise(Y,i,j)
\qquad\text{for some $j$ such that $1 \le j \le\cK_i(Y,k)$.}
\]
\end{lemma}
\begin{proof}
Choose a Jordan basis $\{ u_{pq}^{(r)}\}$ for $M(x)$.  Write $x'(v)$ in this basis:
\[
x'(v) = \sum_{\substack{1 \le q \le k-1 \\ 1 \le r \le y_{i+1,q}}} c_q^{(r)} u_{i+1,q}^{(r)}.
\]
Here, we can take the sum just over $1 \le q \le k-1$ instead of $1 \le q \le n-i$ because if some basis element $u_{i+1,q}^{(r)}$ with $q \ge k$ occurred with nonzero coefficient in the expansion of $x'(v)$, it would follow that $(x')^q(v) = x^{q-1}(x(v)) \ne 0$, a contradiction.

Next, we break up this sum according to whether $r \le y_{i,q+1}$ or $r > y_{i,q+1}$:
\[
x'(v) =
\sum_{\substack{1 \le q \le k-1 \\ 1 \le r \le y_{i,q+1}}}  c_q^{(r)} u_{i+1,q}^{(r)}
+
\sum_{\substack{1 \le q \le \cK_i(Y,k)-1 \\ y_{i,q+1} < r \le y_{i+1,q}}}  c_p^{(r)} u_{i+1,q}^{(r)}.
\]
Here, the second sum is just over $1 \le q \le \cK_i(Y,k)-1$ rather than $1 \le q \le k-1$ because if $\cK_i(Y,k) < q+1 \le k$, then, by the definition of $\cK_i(Y,k)$, we must have $y_{i,q+1} = y_{i+1,q}$.

Consider the vector
\[
v' = v - \sum_{\substack{1 \le q \le k-1 \\ 1 \le r \le y_{i,q+1}}}  c_q^{(r)} u_{i,q+1}^{(r)}.
\]
This vector, like $v$, has nonzero image in $M(\be_i)$.  Moreover, we have
\[
x'(v') = \sum_{\substack{1 \le q \le \cK_i(Y,k)-1 \\ y_{i,q+1} < r \le y_{i+1,q}}}  c_q^{(r)} u_{i+1,q}^{(r)}.
\]
Let $j$ be the largest integer such that some coefficient $c_{j-1}^{(r)}$ with $y_{ij} < r \le y_{i+1,j-1}$ is nonzero.  By relabeling the elements of our Jordan basis, we may assume, in particular, that $c_{j-1}^{(r)} \ne 0$ for $r = y_{ij}+1$.  Then the vectors
\begin{equation}\label{eqn:new-jordan}
x'(v'), (x')^2(v'), \ldots, (x')^{j-1}(v')
\end{equation}
are all nonzero, but $(x')^j(v') = 0$. We can now equip $M(x')$ with a basis as follows: starting from the original Jordan basis $\{u_{pq}^{(r)}\}$, delete the vectors
\[
u_{i+1,j-1}^{(y_{ij}+1)}, u_{i+2,j-2}^{(y_{ij}+1)}, \ldots, u_{i+j-1,1}^{(y_{ij}+1)},
\]
and then add the vectors in~\eqref{eqn:new-jordan}.  More concisely, we are considering the basis
\[
\{ u_{pq}^{(r)} \mid \text{if $p+q = i+j$, then $r \ne y_{ij}+1$} \} 
\cup \{ x'(v'), (x')^2(v'), \ldots, (x')^{j-1}(v') \}
\]
It is straightforward to see that this is a Jordan basis of type $\Raise(Y,i,j)$, as desired.
\end{proof}

\begin{remark}\label{rmk:ses-raise}
In the proof of Lemma~\ref{lem:ses-raise}, we constructed a Jordan basis for $M(x')$, a subset of which constitutes a Jordan basis for $M(x)$.  In other words, $M(x)$ admits a Jordan basis that extends to a Jordan basis for $M(x')$.
\end{remark}

\begin{theorem}
\label{thm:main}
For any $Y \in \bP(\bw)$, we have $\bT(Y) \cong \sT(Y)$.
\end{theorem}

To prove this theorem, we need to recall some general results about algebraic group actions on vector spaces.  Let $H$ be a complex algebraic group acting on a vector space $V$ with finitely many orbits, and let $\cO \subset V$ be an $H$-orbit.  Then one can consider its conormal bundle $N^*\cO \subset V \times V^*$.  According to~\cite{pya:llgf}, there is a natural bijection
\[
Z: \{ \text{$H$-orbits in $V$} \} \to \{ \text{$H$-orbits in $V^*$} \}
\]
determined by the condition that $\overline{N^*(Z(\cO))} = \overline{N^*\cO} \subset V \times V^*$.  Next, according to~\cite[Proposition~7.2]{em:ftimi}, this bijection coincides with the one induced by Fourier transform:
\begin{equation}\label{eqn:em}
\bT(\IC(\cO)) \cong \IC(Z(\cO)).
\end{equation}
On the other hand, for quiver representations, there is another description of the bijection $Z$, due to Zelevinsky.  Consider a pair of quiver representations $(x,y) \in E(\bw) \times E(\bw^*)$.  We can draw this pair as
\[
\begin{tikzcd}
\C^{w_1} \ar[r, shift left, "x_1"] &
\C^{w_2} \ar[r, shift left, "x_2"] \ar[l, shift left, "y_{n-1}"] &
\cdots \ar[r, shift left, "x_{n-1}"] \ar[l, shift left, "y_{n-2}"] &
\C^{w_n} \ar[l, shift left, "y_1"]
\end{tikzcd}
\]
We say that $x$ and $y$ \emph{commute} if
\[
x_i y_{n-i} + y_{n-i-1}x_{i+1} = 0
\qquad\text{for $i = 0, 1, \ldots, n-1$.}
\]
(To make sense of this equation for $i = 0$ or $i = n-1$, we adopt the convention that $x_0 y_n = 0$ and $y_0 x_n = 0$.)  For any $x \in E(\bw)$, let
\[
C(x) = \{ y \in E(\bw^*) \mid \text{$x$ and $y$ commute} \}.
\]
Of course, $C(x)$ is a linear subspace of $E(\bw^*)$.  There is a unique orbit $\cO \subset E(\bw^*)$ such that $\cO \cap C(x)$ is dense in $C(x)$.

\begin{proposition}[Zelevinsky]\label{prop:zelevinsky}
Let $Y \in \bP(\bw)$ and $Y' \in \bP(\bw^*)$.  Then $Z(\cO_Y) = \cO_{Y'}$ if and only if for any point $x \in \cO_Y$, the set $\cO_{Y'} \cap C(x)$ is an open dense subset of $C(x)$ in the Zariski topology.
\end{proposition}

For a proof, see~\cite[Proposition~4.4]{zel:paakl}.

\begin{proof}[Proof of Theorem~\ref{thm:main}]
Suppose $\bw = (w_1,\ldots, w_n)$.  We will prove this latter statement by a double induction on $n$ and on $w_n$.

If $n = 1$, then $\bP(\bw)$ and $\bP(\bw^*)$ each consist of a single element, and the claim is obvious. Suppose from now on that $n > 1$.  If $w_n = 0$, then the space $E(\bw)$ can be identified with $E(w_1, \ldots, w_{n-1})$.  In this case, the theorem holds because it reduces to the claim for triangular arrays of size $n-1$.

Suppose now that $n > 1$ and $w_n  > 0$.  Given $Y \in \bP(\bw)$, let $i_0$ be the smallest integer such that $y_{i_0,n+1-i_0} \ne 0$.  (Since $w_n > 0$, some such $i_0$ exists.)  Define a triangular array $Y'$ by
\[
Y'_{ij} =
\begin{cases}
y_{ij} & \text{if $j < n+1-i$, or if $i < i_0$,} \\
y_{ij} - 1& \text{if $j = n+1-i$ and $i \ge i_0$.}
\end{cases}
\]
Thus, $Y'$ differs from $Y$ only in the last ladder, and we have
\[
y'_{i,n+1-i} - y'_{i-1,n+2-i} =
\begin{cases}
0 = y_{i,n+1-i} - y_{i-1,n+2-i} & \text{if $i < i_0$,} \\
y_{i,n+1-i} - y_{i-1,n+2-i} - 1 & \text{if $i = i_0$,} \\
y_{i,n+1-i} - y_{i-1,n+2-i} & \text{if $i > i_0$.}
\end{cases}
\]
From the formula in Section~\ref{ss:cft}, we see that
\[
\sT(Y) = \sA_{n+1-i_0} \sT(Y').
\]

Choose a point $x \in \cO_Y$, and choose a Jordan basis $\{u_{ij}^{(k)}\}$ for $M(x)$.  Let $\bw' = \udim(Y')$, and identify $\C^{\bw'}$ with the span of
\[
\{ u_{ij}^{(k)} \mid  \text{if $j = n+1-i$, then $j \ge 2$} \}.
\]
This subspace is clearly preserved by $x$.  Let $x' = x|_{\C^{\bw'}}$.  Then $x'$ is of type $Y'$, and the basis above is a Jordan basis for it.  By induction and~\eqref{eqn:em}, we have $\bT(Y') = Z(Y') = \sT(Y')$.  By Proposition~\ref{prop:zelevinsky}, we have
\begin{equation}\label{eqn:ind-dense} 
\cO_{\sT(Y')} \cap C(x') \ne \varnothing.
\end{equation}
The remainder of the proof is broken up into several steps.

\textit{Step 1. The map $p: C(x) \to C(x')$ given by restricting to $\C^{\bw'}$ is surjective.}  Given $\bar x' \in C(x')$, we must show how to extend it to a representation $\bar x$ on $\C^{\bw}$ that commutes with $x$.  To define $\bar x$, we must specify its values on basis vectors of the form $u_{i,n+1-i}^{(1)}$ with $i \ge i_0$.  Choose any vector $v$ in the span of $\{u_{i_0-1,j'}^{(k')} \mid j' \le n+1-i_0 \}$, and then set
\begin{equation}\label{eqn:xbar-defn}
\bar x(u_{i,n+1-i}^{(1)}) = x^{i-i_0} (v).
\end{equation}
We must show that $x$ and $\bar x$ commute, or in other words, that
\begin{equation}\label{eqn:commute}
x \bar x(u_{ij}^{(k)}) = \bar x x(u_{ij}^{(k)})
\end{equation}
for any basis vector $u_{ij}^{(k)}$. If $j \le n -i$, or if $j = n+1-i$ and $k \ge 2$, this holds because $x'$ and $\bar x'$ commute.  On the other hand, if $j = n+1-i$ and $k = 1$, then~\eqref{eqn:commute} follows easily from~\eqref{eqn:xbar-defn}. This completes the proof of Step~1.

For Steps~2--4 of the proof, we let $\bar x'$ be any element of $\cO_{\sT(Y')} \cap C(x')$ (such an element exists by~\eqref{eqn:ind-dense}), and let $\bar x$ be any element of $p^{-1}(\bar x)$ (such an element exists by Step~1).

\textit{Step 2. Notation related to procedure $\sA_{i_0}$.}
Let $\bw' = \udim(Y')$.  From the definition, we see that $\bw = \bw' + (\be_{i_0} + \be_{i_0+1} + \cdots + \be_n)$.  We now define a sequence of intermediate dimension vectors
\[
\bw' = \bw_{i_0-1}, \bw_{i_0}, \bw_{i_0+1}, \ldots, \bw_{n-1}, \bw_n = \bw
\]
by
\[
\bw_m = \bw' + (\be_{i_0} + \be_{i_0+1} + \cdots + \be_m) = \bw_{m-1} + \be_m
\]
Next, define a sequence of integers $q_{i_0-1}, q_{i_0}, \ldots, q_n$ and a sequence of triangular arrays $Z_{i_0-1}, Z_{i_0}, \ldots, Z_n$ with $Z_m \in \bP(\bw_m^*)$ as follows: we first set $q_{i_0-1} = i_0$, and $Z_{i_0-1} = \sT(Y') \in \bP((\bw')^*)$.  If $Z_{m-1}$ and $q_{m-1}$ are already defined, we set
\begin{align}
q_m &= \cK_{n+1-m}(Z_{m-1}, q_{m-1}), \label{eqn:qm-defn} \\
Z_m &= \Raise(Z_{m-1}, n+1-m, q_m) \in \bP(\bw_m^*).\label{eqn:Zm-defn}
\end{align}
This is just an unpacking of the definition of procedure $\sA_{n+1-i_0}$: it is easy to see from the definitions that
\[
\underbrace{\sa \circ \cdots \circ \sa}_{\text{$m+1-i_0$ times}}(\sT(Y'),n+1-i_0,i_0) = (Z_m,n-m,q_m).
\]
In particular, $Z_n = \sA_{n+1-i_0}(\sT(Y')) = \sT(Y)$.

\textit{Step 3. Notation related to the orbit of $\bar x$.}
Identify $\C^{\bw_m}$ with the span of
\[
\{ u_{ij}^{(k)} \mid \text{if $i > m$ and $j = n+1-i$, then $k \ge 2$} \}.
\]
Each of these spaces is preserved by $\bar x$.  They are not preserved by $x$ (except in the extreme cases $m = i_0 - 1$ or $m = n$).  Instead, in general, $x$ restricts to a map $\C^{\bw_m} \to \C^{\bw_{m+1}}$.

Let $\bar x'_m = \bar x|_{\C^{\bw_m}}$.  Then $\bar x'_m \in E(\bw_m^*)$, and the sequence
\[
\bar x' = \bar x'_{i_0-1}, \bar x'_{i_0}, \bar x'_{i_0+1}, \ldots, \bar x'_{n-1}, \bar x'_m = \bar x
\]
can be thought of as ``interpolating'' between $\bar x'$ and $\bar x$.  Note that for each $m \in \{i_0, \ldots, n\}$, there is a short exact sequence
\[
0 \to M(\bar x'_{m-1}) \to M(\bar x'_m) \to M(\be_{n+1-m}) \to 0.
\]
If $m \in \{i_0+1, \ldots, n\}$, then this can be enlarged to a commutative diagram
\begin{equation}\label{eqn:mxbar-commute}
\begin{tikzcd}
0 \ar[r] &
  M(\bar x'_{m-2}) \ar[d, "x"'] \ar[r] &
  M(\bar x'_{m-1}) \ar[d, "x"'] \ar[r] &
  M(\be_{n-m}) \ar[d, "\wr"', "x"] \ar[r] & 0 \\
0 \ar[r] &
  M(\bar x'_{m-1}) \ar[r] &
  M(\bar x'_m) \ar[r] &
  M(\be_{n+1-m}) \ar[r] & 0
\end{tikzcd}
\end{equation}  
Let $Z'_m \in \bP(\bw_m^*)$ be the label of the orbit of $\bar x'_m$.  By Lemma~\ref{lem:ses-raise}, there is an integer $q'_m$ such that
\begin{equation}\label{eqn:Zmp-defn}
Z'_m = \Raise(Z'_{m-1}, n+1-m, q'_m).
\end{equation}
Note that $Z'_n$ is the label of the orbit of $\bar x$.

\textit{Step 4. For $m \in \{i_0, i_0+1, \ldots, n\}$, we have $q'_m \le q_m$ and $Z'_m \le Z_m$.}  We proceed by induction on $m$.  For $m = i_0$, the integer $q_{i_0} = \cK_{n+1-i_0}(Z_{i_0-1}, i_0)$ is the largest integer $q$ such that $\Raise(Z_{i_0-1},n+1-i_0,q)$ is defined. Since $Z'_{i_0-1} = Z_{i_0-1}$, and since $\Raise(Z_{i_0-1},n+1-i_0,q'_{i_0})$ is also defined, we have
\[
q'_{i_0} \le q_{i_0}.
\]
The triangular arrays $Z'_{i_0}$ and $Z_{i_0}$ differ only in the $(n+1-i_0)$th chute.  It is clear from~\eqref{eqn:Zm-defn}, \eqref{eqn:Zmp-defn}, and the definition of the partial order that $Z'_{i_0} \le Z_{i_0}$.

Now suppose that $m > i_0$, and that $q'_{m-1} \le q_{m-1}$.  By Remark~\ref{rmk:ses-raise}, we may choose a Jordan basis for $M(\bar x'_{m-2})$ that extends to a Jordan basis for $M(\bar x'_{m-1})$.  The latter adds one extra basis element $u$ with the property that $(\bar x'_{m-1})^{q'_{m-1}}(u) = 0$.  The commutative diagram~\eqref{eqn:mxbar-commute} shows that $M(\bar x'_m)$ is spanned by $M(\bar x'_{m-1})$ and $x(u)$.  Let $k$ be the smallest integer such that $(\bar x'_m)^k (x(u)) = 0$.  As $(\bar x'_m)^{q'_{m-1}}(x(u)) = x (\bar x'_{m-1})^{q'_{m-1}}(u) = 0$, we clearly have $k \le q'_{m-1} \le q_{m-1}$.  By Lemma~\ref{lem:ses-raise} and the definition of $q'_m$, we have
\[
q'_m \le \cK_{n+1-m}(Z'_{m-1}, k) \le k \le q'_{m-1}.
\]
We will now show that $q'_m \le q_m$.  If $q_m > k$, the claim is obvious.  Suppose instead that $q_m \le k$. Then we can replace~\eqref{eqn:qm-defn} by
\[
q_m = \cK_{n+1-m}(Z_{m-1},k).
\]
Since $k \le q'_{m-1}\le q_{m-1}$ as well, the first $k-1$ entries of the $(n+2-m)$th chutes of $Z'_{m-1}$ and $Z_{m-1}$ agree (and coincide with the corresponding entries of $\sT(Y')$).  Of course, the $(n+1-m)$th chutes of $Z'_{m-1}$ and $Z_{m-1}$ also agree with the $(n+1-m)$th chute of $\sT(Y')$.  Since $Z'_{m-1}$ and $Z_{m-1}$ agree on all entries relevant to the computation of $\cK_{n+1-m}({-},k)$, we conclude that $\cK_{n+1-m}(Z'_{m-1}, k) = \cK_{n+1-m}(Z'_{m-1}, k)$, and hence that $q'_m \le q_m$, as desired.

It remains to show that $Z'_m \le Z_m$.  The triangular arrays $Z'_m$ and $Z_m$ differ from $Z'_{m-1}$ and $Z_m$, respectively, only in the $(n+1-m)$th chute.  Since $Z'_{m-1} \le Z_{m-1}$, in order to compare $Z'_m$ and $Z_m$, we need only compare their $(n+1-m)$th chutes.  It is clear from~\eqref{eqn:Zm-defn} and~\eqref{eqn:Zmp-defn} that
\[
\sum_{p = 1}^j (Z'_m)_{n+1-m,p} \ge \sum_{p=1}^j (Z_m)_{n+1-m,p}
\]
for all $j$ (indeed, they are equal unless $q'_m < j < q_m$, in which case the left-hand side is larger by $1$).  We conclude that $Z'_m \le Z_m$, as desired.

\textit{Step 5. Conclusion of the proof.}
By~\eqref{eqn:em} and Proposition~\ref{prop:zelevinsky}, $\cO_{\bT(Y)} \cap C(x)$ is a Zariski-open dense subset of $C(x)$.  The surjective linear map $p: C(x) \to C(x')$ is an open map, so $p(\cO_{\bT(Y)} \cap C(x))$ is a Zariski-open dense subset of $C(x')$.  The same holds for $\cO_{\sT(Y')} \cap C(x')$ (see the remarks preceding~\eqref{eqn:ind-dense}, so
\[
p(\cO_{\bT(Y)} \cap C(x)) \cap \cO_{\sT(Y')} \cap C(x') \ne \varnothing.
\]
Choose $\bar x'$ in this set, and then choose $\bar x \in \cO_{\bT(Y)} \cap C(x)$.  Apply Steps~2--4 to these elements.  From Steps~2 and~3, we have $Z_n = \sT(Y)$ and $Z'_n = \bT(Y)$.  Step~4 then tells us that
\[
\bT(Y) \le \sT(Y).
\]
This inequality holds for all $Y \in \bP(\bw)$.  But since $\bT$ and $\sT$ are both bijections, this inequality actually implies that $\bT(Y) = \sT(Y)$ for all $Y$.
\end{proof}

\begin{corollary}\label{cor:main}
For all $Y \in \bP(\bw)$, we have $\sT(Y) = \sICFT(Y)$.
\end{corollary}
\begin{proof}
The map $\sICFT$ is the inverse of $\sT$, but by Theorem~\ref{thm:main} and~\eqref{eqn:bt-inv}, $\sT$ is an involution.
\end{proof}


\end{document}